\documentclass[a4paper,11pt]{article}
\usepackage[latin1]{inputenc}
\usepackage[T1]{fontenc}
\usepackage[british,UKenglish,USenglish,american]{babel}
\usepackage{amsmath}
\usepackage{amsfonts}
\usepackage{hyperref}
\usepackage[T1]{fontenc}
\usepackage[latin1]{inputenc}
\usepackage{ntheorem}
\usepackage{theorem}
\usepackage{graphics}
\usepackage{makeidx}
\usepackage{fancyhdr}
\usepackage{hyperref}
\usepackage{verbatim}
\usepackage{bbm}
\usepackage{bm}
\usepackage{amssymb}
\usepackage{color}
\numberwithin{equation}{section}
\usepackage{amssymb}
\newcommand{\be}{\begin{equation}}
\newcommand{\ee}{\end{equation}}
\newcommand{\benn}{\begin{equation*}}
\newcommand{\eenn}{\end{equation*}}
\newcommand{\bea}{\begin{eqnarray}}
\newcommand{\eea}{\end{eqnarray}}

\newcommand{\beann}{\begin{eqnarray*}}
\newcommand{\eeann}{\end{eqnarray*}}

\newtheorem{theorem}{Theorem}[section]

\newtheorem{corollary}[theorem]{Corollary}
\newtheorem{lemma}[theorem]{Lemma}

\newtheorem{definition}[theorem]{Definition}
\theorembodyfont{\rm}

\newtheorem{remark}[theorem]{Remark}

\newtheorem{assumptions}[theorem]{Assumptions}
\newtheorem{ansatz}[theorem]{Ansatz}

\newcommand{\qed}{\hfill $\Box$\smallskip}
\newcommand{\E}{\noindent{$\mathbb{E}$ \ }}

\def\R{\mathbb{R}}

\def\P{\mathbb{P}}
\def\E{\mathbb{E}}

\def\P{\mathbb{P}}

\def\cF{\mathcal{F}}

\def\cH{\mathcal{H}}

\def\cL{\mathcal{L}}

\def\cN{\mathcal{N}}
\def\cO{\mathcal{O}}

\def\cS{\mathcal{S}}

\def\txtd{{\textnormal{d}}}

\allowdisplaybreaks

\title{Finite-time Lyapunov exponents for SPDEs with fractional noise}
\author{Alexandra Blessing (Neam\c tu)~\thanks{Alexandra Blessing (Neam\c tu).  Department of Mathematics and Statistics, University of Konstanz,
Universit\"atsstra\ss{}e~10, 78464 Konstanz, Germany. E-Mail: alexandra.neamtu@uni-konstanz.de}~~~~and~~~~Dirk Bl\"omker\thanks{Dirk Bl\"omker. Institut f\"ur Mathematik, Universit\"at Augsburg, Universit\"atsstra{\ss}e 12, 86135 Augsburg, Germany.~E-Mail: dirk.bloemker@math.uni-augsburg.de} }
\begin{document}
\maketitle

\begin{abstract}
We estimate the finite-time Lyapunov exponents for a stochastic partial differential equation driven by a fractional Brownian motion (fbm) with Hurst index $H\in(0,1)$ close to a bifurcation of pitchfork type.
 We characterize regions depending on the distance from bifurcation, the Hurst parameter of the fbm and the noise strength where finite-time Lyapunov exponents are positive and thus indicate a change of stability. The results on  finite-time Lyapunov exponents are novel also for SDEs perturbed by fractional noise.

\end{abstract}

\section{Introduction}

The main goal of this work is to 
provide a tool for a bifurcation analysis
for general SPDEs of the type
\begin{align}\label{spde1}
\begin{cases}
\txtd u = [A u + \nu u +\cF(u)]~\txtd t +\sigma \txtd W^H_t \\
u(0)=u_{0},
\end{cases}
\end{align}
where $A$ is a linear operator with a one dimensional kernel, the parameter $\nu\in\R$ shifts the spectrum of $A$, $\cF$ is a stable cubic nonlinearity and $\sigma>0$ denotes the intensity of the infinite-dimensional noise, which is given by a Hilbert-space valued fractional Brownian motion $(W^H(t))_{t\in[0,T]}$ with Hurst index $H\in(0,1)$. The results obtained in this work provide a novel bifurcation analysis also for SDEs. Due to the non-Markovianity of the noise, to our best knowledge no methods have been developed so far for the computation of (finite-time) Lyapunov exponents. The work~\cite{KLN:22} investigates related aspects using early-warning signs to detect changes of stability in a finite-dimensional slow-fast system perturbed by additive fractional noise where $H>1/2$. The analysis of the finite-time Lyapunov exponents for the SDE 
\begin{align}\label{sde:p} 
	\txtd x = (\nu x -x^3)~\txtd t+\sigma~\txtd W_t
\end{align}
driven by a Brownian motion around its unique random equilibrium  splits into different cases.
For $\nu<0$ one  relies on deterministic stability and shows that FTLE are negative with probability one. For $\nu>0$ one essentially relies on the structure of the invariant measure, where one needs that this has to be close to zero on a set of positive probability. Together with an estimate of the paths of the Brownian motion on a finite time horizon, the estimates for the FTLE on a set of positive probability become quite accessible~\cite{CDLR17}. 

All these arguments break down in the context of a fractional Brownian motion. Obviously, the events describing the location of the random equilibrium at the initial time and the finite-time estimates of the paths of the fractional Brownian motion are in general not independent anymore. 
In order to overcome this issue we focus on deterministic initial data and expect to obtain positive finite-time Lyapunov exponents for the SDE~\eqref{sde:p} with a fractional Brownian motion.
This provides a novel insight in the dynamics of SDEs perturbed by fractional noise. 
In order to analyze the SPDE case as in~\cite{BlNe:23} we rely on the approximation with amplitude equations (AE) also studied  in~\cite{BlNe:22}. These results  are valid for the full range of Hurst indices $H\in(0,1)$ and are applicable if the noise is given by a trace-class fractional Brownian motion. 
This assumption was crucial  in order to to control the approximation order, which further reflects the time-scale on which we observe the finite-time Lyapunov exponents.

\paragraph{Main Results.}
Relying on the self-similarity of the fractional Brownian motion
\begin{align*}
	& W^H(T\varepsilon^{-2}) \stackrel{\mbox{law}}{=} W^H(T)\varepsilon^{-2H} ~~\text{ and of the derivative   }	~~ \dot{W}^H(T\varepsilon^{-2}) \stackrel{\mbox{law}}{=} \varepsilon^{2-2H}\dot{W}^H(T),
\end{align*}
the approximation of the SPDE~\eqref{spde} via amplitude equations was derived in~\cite{BlNe:22}. 
Based on this result and using the theory of finite-time Lyapunov exponents we show the following behavior of FTLE from~\eqref{spde1} near a change of stability. The precise statement is based on the interplay between the distance towards the bifurcation, intensity of the noise and Hurst parameter of the fbm:



\begin{itemize}
    \item[I.] {\em Before the bifurcation, $\nu<0$.}
    The solutions of~\eqref{spde1} are stable for all $\sigma$ with probability one. Therefore we show that all FTLEs are negative with probability one.

    \item[II.] {\em After the bifurcation, moderate noise strength,   $0<\sigma\approx \nu^{H+\frac12} \ll1$.}
    Here we have instability, namely there is a solution for which the finite-time Lyapunov exponent $ \lambda_T>0 $ is positive with positive probability,  for times of order $1/\nu^{H+\frac12}$.  Due to the nature of the approximation result (Theorem~\ref{thm:aprox}), our result is applicable for times $T$ of order $1/\nu^{H+1/2}$ but for technical reasons not up to $0$.  This means there is a $\nu$ dependent interval, where we can prove that the FTLE is positive and this interval contains values of the type $C/{\nu^{H+1/2}}$ for some values of the constant $C>0$. \\
    The proof of this statement  relies on the approximation with the amplitude equation $\txtd  b =( b -b^3)~\txtd T +\frac{\sigma}{\nu^{\frac12+H}} ~\txtd \beta^H(T)$, where $(\beta^H(T))_{T\geq 0}$ is a fractional Brownian motion. A novel argument in the proof of this statement is to start the amplitude equation in a suitably rescaled deterministic initial data for which the solution of the amplitude equation is located in a ball around zero with positive probability.  
    
\item[III.] {\em After the bifurcation, small noise strength, $0<\sigma \ll \nu^{H+\frac12}\ll 1$} 
Similar to case II.\ we observe instability using the amplitude equation $\txtd b = (b - b^3)~\txtd T$. 
    
    \item[IV.] {\em At the bifurcation, $0\leq \nu^{H+\frac12} \ll \sigma \ll1$.}
    Here we have
    stability, meaning that $\lambda_T<0$  for all solutions with positive probability. The proof of this statement relies on the approximation with 
    the amplitude equation  of the type  $\partial_T b = -b^3\txtd T + \txtd \beta^H(T)$, where $(\beta^H(T))_{T\geq 0}$ is a fractional Brownian motion. 
\end{itemize}

This analysis provides novel results for finite-time Lyapunov exponents with fractional noise and also improves the results in~\cite{BlNe:23} for the Brownian motion. More precisely in order to prove the statement II. we show that it is not necessary to start the solutions of the SPDE in the rescaled attractor of the SDE but in an arbitrary rescaled small deterministic initial condition. 
This information allows us to apply a support theorem~\cite[Proposition 5.8]{HaOh:07} from which we can further find positive finite-time Lyapunov exponents for SDEs with fractional additive noise and stable cubic nonlinearities with positive probability.

We transfer these results to the finite-time Lyapunov exponents of the SPDE~\eqref{spde} approximating it with an amplitude equation as established in Theorem~\ref{thm:aprox}. A technical step is to quantify the dependence on $H$ of the approximation error between the linearization of the SPDE and the linearization of the amplitude equation around a solution which satisfies the support theorem. 

As one of the main technical novelties in order to prove IV. we establish a lower bound on the probability that the amplitude equation is small for lots of times. This argument does not rely on stationary solutions and Birkhoff's ergodic theorem as for the case of a Brownian motion~\cite[Lemma 4.2]{BlNe:23}, as this is not available for solutions of SDEs driven by fractional noise. In all cases we compute an error term between the two linearizations for $H\in(0,1)$ obtaining different bounds depending on the range of $H$.

The paper is organized in the following way. Section \ref{assumptipns} states all necessary assumptions and establishes the setting we are working in. 
In Section \ref{sec:ex} we give a short remark about existence and uniqueness  of solutions for our SPDE.
The properties of FTLE for the corresponding AE are studied in Section \ref{four}, while Section \ref{sec:approx} provides the key approximation result for AE. The final Section \ref{six} states the main results in full details and provides their proofs.

\section{Setting and Assumptions}\label{assumptipns}
We work in the following setting. We let $\cH$ stand for a separable Hilbert space and consider the SPDE driven by an Hilbert-space valued fractional Brownian motion $(W^H(t))_{t\in[0,T]}$ with Hurst index $H\in(0,1)$
\begin{align}\label{spde}
\begin{cases}
\txtd u = [A u + \nu u +\cF(u)]~\txtd t +\sigma \txtd W^H_t \\
u(0)=u_{0}\in \cH.
\end{cases}
\end{align}
We make the following standard assumptions on the linear operator $A$ and on the cubic non-linearity $\cF$. 
\begin{assumptions}\label{a} (Differential operator $A$) The linear operator $A$ generates a compact analytic semigroup $(e^{tA})_{t\geq 0}$ on $\cH$. 
 Moreover, it is symmetric and non-positive and has a one-dimensional kernel which we denote by $\cN$.  
 We define  the orthogonal projection  $P_c$ onto $\cN$, set $P_s=\text{Id}-P_c$ and obtain that $\cH=\cN\oplus \cS$, where $\cS$ stands for the range of $P_s$. 
 The semigroup is exponentially stable on $P_s\cH$ which means that there exists $\mu>0$ such that 
\[ 
\|e^{tA} P_s\|_{\cL(\cH)}\leq e^{-t\mu},~~\text{ for all } t\geq 0. 
\]
We further define the spaces $\cH^\alpha=D((1-A)^\alpha)$ for $\alpha\geq 0$ endowed with the norm $\|\cdot\|_\alpha=\|(1-A)^\alpha\cdot \|$ and scalar product $\langle u, v \rangle_\alpha=\langle (1-A)^\alpha u,(1-A)^\alpha v \rangle$ and set $\cH^{-\alpha}=(\cH^\alpha)^*$ the dual of $\cH^\alpha$. It is well-known that $(e^{tA})_{t\geq 0}$ is an analytic semigroup on $\cH^\alpha$ for every $\alpha\in\mathbb{R}$. Finally, we have that $\cN\subset \cH^\alpha$ for all $\alpha>0$ since  $(1-A)^\alpha\cN=\cN$. 
\end{assumptions}

Under our assumptions we have for some constant $C>0$ depending on $\alpha>0$ that
$\|A^\alpha P_su\| \geq C\|P_su\|$ 
for all $u\in\cH$, which we use frequently.

\begin{assumptions}\label{f}(Nonlinearity)
We assume that there exists a Banach space $X$ such that 
\[
\cH^{\alpha}  \subset X \subset \cH
\]
for $\alpha\in(0,1/2)$ with continuous and dense embeddings.
Moreover, the mapping $\mathcal{F}: X\to X^*\subset \cH^{-\alpha}$
is a stable cubic (i.e.\ trilinear) nonlinearity with
\begin{equation}
\label{e:stableF}
\langle \cF(u)-\cF(v), u-v\rangle 
\leq - c\|u-v\|_X^4  ,\quad \mbox{ for } u,v \in X.
\end{equation} 
\end{assumptions}
Let us remark that we can allow terms like $C\|u-v\|^2$ on the r.h.s.\ of \eqref{e:stableF}, but we can always modify the linear term to remove these terms.

%
\begin{assumptions}\label{s}(Noise)
We assume that $(W^H(t))_{t\in[0,T]}$ is a trace-class fractional Brownian motion on $\cH$.  This assumption was made in~\cite{BlNe:22} and in particular it implies that the stochastic convolution 
\[ Z(t)=\int_0^t e^{A(t-s)}~\txtd W_s \]
is well-defined and $Z\in C([0,T];\cH^\alpha)$ for $\alpha<H$. 
\end{assumptions}

\begin{remark}
For $\cF$ we can show the following sign condition. For any positive $\delta>0$ there is a constant $C>0$ depending on $\delta$ such that for all $u, z\in X$
\begin{equation}
\label{e:boundF}
\langle\cF(u+z),u\rangle 
\leq -c\|u+z\|^4_X + C\|u+z\|^3_X\|z\| 
\leq - \delta\|u\|^4_X + C_\delta \|z\|_X^4.
\end{equation}

As $\cF$ is trilinear we readily have that $\cF$ is Fr\'echet-differentiable with 
\[
D\cF(u)[h]= \cF(u,u,h)+\cF(u,h,u)+\cF(h,u,u).
\] 
Moreover, for $u,h\in  X$ we obtain due to~\eqref{e:boundF}
\begin{equation}
    \label{e:sign}
    \langle D\cF(u)h,h\rangle 
= \lim_{t\to 0} \frac1t \langle \cF(u+th)-\cF(u), h\rangle
\leq - \lim_{t\to 0} \frac1{t^2} \|th\|^4
= 0 .
\end{equation}
\end{remark}

Let us remark additionally, that we use an estimate like $D\cF_c(b)\leq -c b^2$ in our proof that arises from the one-dimensionality of $\mathcal{N}$.   
In order to remove the condition that $\cN$ is one-dimensional, we will need that $\cF_c$ is a genuine non-degenerate cubic term with an analogous estimate. 
We refer to potentials similar to~\cite{GeTs:22}, where for example $\cF(b)=-cb|b|^2$ was treated.\\

In~\cite{BlNe:22} we used the following definition of the $\cO$ notation. 
\begin{definition}
We say that a term $F_{\varepsilon}=\cO(f_{\varepsilon})$ if and only if there exist positive $\varepsilon$-independent constants $C$ and $\varepsilon_0$ such that $|F_{\varepsilon}|\leq C f_{\varepsilon}$ for all $\varepsilon\in(0,\varepsilon_0]$.

For a random quantity we write $F_{\varepsilon}=\cO(f_{\varepsilon})$ if the above the statement holds true  on a set with probability going to 1 if $C\to \infty$. 
\end{definition}

\begin{assumptions}\label{o:z}
For the stochastic convolution, 
we have for every small $\kappa>0$
\begin{align}\label{scaling:z}P_s Z = \cO(T^{\kappa})\qquad\text{and}\qquad 
P_c Z = P_c W^H = \cO( T^{H})
\end{align}
uniformly in $T$ on any interval $[0,T_0]$ in the space $X$. 
\end{assumptions}
\begin{remark}
These bounds were obtained  using the scaling properties of the fractional Brownian motion and the factorization method in~\cite[Appendix B]{BlNe:22}. We have 
 for every small $\kappa>0$
\begin{align*}
\sup\limits_{t\in[0,T_0\varepsilon^{-2}]}\|P_c W^H(t)\| \stackrel{\mbox{law}}{=} \varepsilon^{-2H} \sup\limits_{T\in[0,T_0]} \|P_cW^H(T)\| =\mathcal{O}(\varepsilon^{-2H-\kappa})
 \end{align*}
 with probability almost $1$,
 whereas, with high probability
 \begin{align*}
\sup\limits_{t\in[0,T_0\varepsilon^{-2}]} \|P_sZ(t)\|=\cO(\varepsilon^{-\kappa}).
 \end{align*}
\end{remark}


\section{Existence of solutions }\label{sec:ex}
The existence of solutions of SPDE with additive fractional noise and stable cubic nonlinearities was established in~\cite[Theorem 4.3]{MaSc:04}. In order to obtain some regularity properties of the solution we briefly sketch an alternative proof similar to the case of the Brownian motion~\cite{BlNe:23} which relies  on the Galerkin method and the standard transformation $w=u-Z$ that solves the random PDE
 \[
\partial_t w = Aw +\nu w +\cF(w+Z).
 \]
 For this equation one can apply classical pathwise existence results,
 see for example \cite{Te:97,Rou:13}. This is based on \eqref{e:boundF}
giving regularity in $L^4(0,T,X)$, together with the compact embedding of $X$ into $\cH^{1/2}$ and Aubin-Lions Lemma.

For initial conditions in $\cH$ and $Z$ being a continuous stochastic process with values in $\cH^\alpha\subset X$ this shows global 
existence of solutions such that for all $T>0$
\[
u-Z \in L^2(0,T,\cH^{1/2}) \cap  C^0([0,T],\cH)\cap L^4(0,T,X) 
\]
which also implies some regularity of  $\partial_t(u-Z)$ as 
$A(u-Z) \in L^{2}(0,T,\cH^{-1/2})$ and 
$\cF(u) \in L^{4/3}(0,T,X^*)$.

The pathwise uniqueness of solutions follows immediately from \eqref{e:stableF}. For the 
difference $d=u_1-u_2$ of two solutions 
$u_1$ and $u_2$ satisfying
\[
\partial_t d = Ad +\nu d +\cF(u_1)-\cF(u_2)
\]
we only need 
the differentiability of the $\cH$-norm to conclude
\[\partial_t \|d\|^2 = \langle Ad + \nu d +\cF(u_1)-\cF(u_2), d \rangle \leq \nu \|d\|^2.
\]
The differentiability of the norm follows, as  we have \[d=(u_1-Z)-(u_2-Z) \in L^2(0,T,\cH^{1/2}) \cap  L^\infty(0,T,\cH)
\]
by standard parabolic regularity together with $d \in L^4(0,T,X)$ and 
$\cF(u_i) \in L^{4/3}(0,T,X^*)$.

With the arguments sketched above one can prove the following theorem, which we state without proof.
\begin{theorem}
\label{thm:exSPDE}
Let Assumptions~\ref{a},~\ref{f},~\ref{s} be satisfied.
Then for all initial conditions $u_0\in\cH$ there is a unique (up to global null sets)
stochastic process $u$ with continuous paths in $\cH$, 
which is a weak solution of~\eqref{spde}
and satisfies for all $T>0$ 
\[
u-Z \in L^2(0,T,\cH^{1/2}) \cap  C^0([0,T],\cH) 
\cap L^4(0,T,X). 
\]
\end{theorem}

\subsection{Finite-time Lyapunov-exponents}

The linearization $D_{u_0} u(t,\omega,u_0)$ of~\eqref{spde1} around a solution $u(t,\omega,u_0)$ with deterministic initial condition $u_0$ is defined as the solution $v(t,\omega,u_0,v_0)$ of  the linear PDE called also the first variation equation, which due to the additive structure of the noise is given by see~\cite{BlEnNe:21,BlNe:22}
\begin{align}\label{linearization}
\begin{cases}
    \txtd v =  [Av + \nu v + D \cF(u) v]~\txtd t \\
    v(0)=v_0.
    \end{cases}
\end{align}
\begin{remark}
The Fr\'echet differentiablity of the solution operator $u_0\mapsto u(t,\omega,u_0)$ follows regarding that $u\in L^2(0,T;\cH^{1/2})$ due to~\cite[Lemma 4.4]{deb}.
\end{remark}
For $t>0$ we denote the random solution operator $U_{u_0}(t):\cH\to\cH$ such that $v(t)=U_{u_0}(t)v_0$, where $v$ is a solution of \eqref{linearization} given the initial condition $v_0\in \cH$.

\begin{remark}
Note that for any solution $u\in L^4(0,T,X)$ we have $\cF(u)\in L^{4/3}(0,T,X^*) \subset L^{4/3}(0,T,{\cH}^{-\alpha})$.
We can now use pathwise deterministic theory for linear PDEs. 
For example Galerkin methods show that for given $v_0\in\cH$ there is an (up to global null sets) unique stochastic process $v$ with continuous paths in $\cH$ and $v \in L^2(0,T;\cH^{1/2})$ for all $T>0$ that solves \eqref{linearization}. 
\end{remark}

We define the finite-time Lyapunov exponent as in~\cite{BlEnNe:21, BlNe:22}.

\begin{definition}{\em (Finite-time Lyapunov exponent)}. Let $t>0$ be fixed. We call a finite-time Lyapunov exponent for a solution $u$ of the SPDE with (random) initial condition $u_0$ 
\begin{equation}\label{ftle}
    \lambda_t(u_0):=\lambda(t,\omega,u_0) =\frac{1}{t} \ln \left( \| U_{u_0}(t) \|_{\cL(\cH)}\right).
\end{equation}
\end{definition}

From the definition it is clear that finite-time Lyaponov exponents measure local expansion rates of nearby solutions. Negative finite-time Lyapunov exponents indicate attraction whereas positive ones indicate that nearby solutions tend to separate on a finite time horizon.

\begin{remark}
\label{propLyap}
We can compute $\|U_{u_0}\|_{\cL(\cH)}$ as follows
\begin{eqnarray*}
\| U_{u_0}(t) \|_{L(\cH)}
&=& \sup\{ \|v(t)\| / \|v(0)\| \ : \  v \text{ solves \eqref{linearization} with }v(0)\not=0 \} \\
&=& \sup\{ \|v(t)\|  \ : \  v \text{ solves \eqref{linearization} with }\|v(0)\|=1 \}.
\end{eqnarray*}
\end{remark} 

\begin{remark} 
Let us comment on the following.
   \begin{itemize}
       \item [1)] The limit $t\to\infty$ in~\eqref{ftle} gives an asymptotic Lyapunov exponent. For SDEs perturbed by Brownian motion, one can show based on ergodic theorems that the asymptotic Lyapunov exponent exists and is independent of $\omega$. However, estimating this quantity for SPDEs is a highly challenging task, compare~\cite{GeTs:22,BeBlPs:22}.
       \item [2)] Both finite time and asymptotic Lyapunov exponents have not been investigated for S(P)DEs with fractional noise so far. 
       \item [3)] For technical reasons which will be explained in Section~\ref{four} we restrict ourselves to deterministic initial data $u_0$. The independence of $u_0$ from the fractional Brownian motion helps us.
   \end{itemize} 
\end{remark}


\section{Finite-time Lyapunov exponents for amplitude equations with fractional noise}\label{four}

In this setting $\sigma>0$ and $\nu\geq 0$ are fixed quantities that depend on a small parameter $\varepsilon$ and we assume the following upper bound:
\[ 
\nu=\cO(\varepsilon^2) \text{ and } \sigma=\cO(\varepsilon^{2H+1}). 
\]
Using the cubic nonlinearity and the interplay between $\nu$ and $\sigma$ we later obtain amplitude equations of two types. In case of $  \sigma \nu^{-\frac12-H}=\cO(1)$  we have
\begin{align}\label{a1}
    \txtd b = (b+\cF_c(b))~\txtd T +\frac{\sigma}{\nu^{\frac12+H}}~\txtd \beta^H_{\nu^{H}}(T)
\end{align}
whereas in the case $\nu^{H+1/2}\ll \sigma$
\begin{align}\label{a2}
\txtd b = \cF_c(b)~\txtd T +\txtd\beta^H_{\sigma^H}(T),
\end{align}
where $(\beta^H_\gamma(T))_{T\in[0,T_0]}$ is an $\cN$-valued fractional Brownian motion rescaled in time by a factor $\gamma$, meaning that $\beta^H_\gamma(T)=\gamma^{2H} \beta^H(T\gamma^{-2})$ for some fractional Brownian motion $\beta^H$.

\begin{lemma}{\em (Positive FTLE for~\eqref{a1})}\label{l1}
Fix $T_0>0$.
If $  \sigma \nu^{-\frac12-H}=\cO(1)$ 
then  there is an $\eta>0$ 
such that for 
$|b_0|<\eta$ sufficiently small,
then 
\[ \mathbb{P}\left( \lambda_T (b_0) \geq \frac{1}{4}  \right)>0 \quad\text{for all } T\in[0,T_0] . 
\]
\end{lemma}
\begin{proof}
Let us think of $b_0$ being random and
introduce the sets 
\[A_1:=\{ \omega\in\Omega: b_0(\omega) \in(-\eta,\eta) \},\]
\[A_2:=\{ \omega\in\Omega: \sup\limits_{t\in[0,T]} \beta^H_{\nu^H}(t)\leq \frac{\eta}{2} \}.
\] 
Here $b_0$ is the initial data of~\eqref{a1}, $\eta>0$ is small and $\beta^H_{\nu^H}$ is just a rescaled fbm. Note that due to the non-Markovianity of the noise the events $A_1$ and $A_2$ are in general not independent as in the Brownian case. 
Therefore in order to guarantee that $\P(A_1\cap A_2)>0$ we consider only deterministic initial conditions $b_0$ (independent of the fractional Brownian motion) and restrict ourselves to $\omega\in\Tilde{\Omega}:=A_2$ where the fractional Brownian motion remains small for a finite-time horizon. In this case we derive for $\omega\in \tilde{\Omega}$ that
\[ 
|b(T)| \leq \Big(1+\frac{\sigma}{\nu^{\frac12+H}}\Big)\eta e^T<\delta \text{ for all }T\in[0,T_0]. 
\]
This statement follows using the equation~\eqref{a1}, the stable cubic nonlinearity and the fact that $\omega\in \tilde{\Omega}$. Alternatively one can apply the support theorem~\cite[Proposition 5.8]{HaOh:07} which states that for every deterministic initial data and every path of the fbm, the solution of~\eqref{a1} will reach a small neighborhood of the origin with positive probability. This information combined with the fact that the noise remains bounded on the finite-time interval $[0,T]$ which we consider, provides the upper bound on $b$.  
Analogously to the case of the Brownian motion, such a bound on the solutions implies the positivity of the FTLEs on the set of positive probability $\tilde{\Omega}$ since
\[
\lambda_T (b_0) =\frac{1}{T} \ln \Bigg( \exp \Bigg( T + \int_0^T D\cF_c(b(s,\omega))~\txtd s\Bigg) \Bigg) \geq \frac{1}{4},  
\]
choosing $\delta:=\frac12$.
\qed
\end{proof}

For our result we cannot rely on Birkhoffs ergodic theorem, instead we use a simple argument similar to~\cite[Theorem 3.4]{Bl:07} to show that the set of times for which the amplitude equation is small has a small probability. There it was used for $\nu\leq0$ in order to show pattern formation below the threshold of stability.  

\begin{lemma}{\em (Negative FTLE for~\eqref{a2})}\label{l2}
Suppose that $\nu^{H+1/2}\ll \sigma$, fix $T_0>0$, and consider a solution of \eqref{a2}. Then for all $\theta>0$
there is a small positive time  $T_\theta\to0$ for $\theta\to0$ such that  uniformly
for all $T\in[T_\theta,T_0]$ 
\[ \mathbb{P}\left( \lambda_T(b_0) <-c\theta^2\right) \to 1 \text{ as } \theta\to 0.
\] 
\end{lemma}

 \begin{corollary}
 Under the assumptions of Lemma \ref{l2} for a fixed $T$ we have $\mathbb{P}\left( \lambda_T(b_0) < 0 \right)=1.$
  \end{corollary}
 For the proof just note that  we can choose $\theta_0$ such that $T_\theta<T$ for all $\theta\in(0,\theta_0)$. But now we have 
 \[\mathbb{P}\left( \lambda_T(b_0) < 0 \right) \geq \mathbb{P}\left( \lambda_T(b_0) <-c\theta^2\right) \to 1
 \quad \text{for } \theta \to 0. \]

   \begin{remark}
    In the proof we will see that $T_\theta=2\sqrt{p_\theta}$  which we will choose later by  \eqref{e:ptheta}.
    Let us remark that we expect 
    $p_\theta\approx c\theta$, 
    as we can approximate the probability by $2p(s) \theta$ where $p(s)$ is the value of the density for $b(s)$ in $0$.
    Thus we could work out a qualitative 
    bound if we have more knowledge about
    the density of $b$.  
   \end{remark}

\begin{proof}
 In order to prove the statement  we have to make sure that the solution $b$ does not stay too close to zero for too many times.
   The linearization of~\eqref{a2} around a solution $b$ entails in this case
\[ \txtd \varphi =D\cF_c(b)~\txtd t. \]
We have to show that there exists a constant $\tilde{c}>0$ such that
\[
\lambda_T(b_0(\omega)) =\frac{1}{T} \ln \exp \Big( \int_0^T D\cF_c(b(s,\omega ))~\txtd s\Big) < -\tilde{c}<0~\quad \text{ for some }\omega \text{ and } T.
\]
Following~\cite[Theorem 3.4]{Bl:07} we define the set of times for which the solution of the amplitude equation is small, i.e. for $T_0>0$ and $\theta>0$
   \[ \mathcal{T}_\theta(T_0):= |\{ s\in[0,T_0] : |b(s)|\leq \theta \} | 
   \]  
   and notice that
   \[  \mathcal{T}_\theta(T_0) =\int_0^{T_0} \mathbbm{1}_{\{|b(s)|\leq \theta\}}~\txtd s.
   \]
Since $\mathcal{T}_\theta\in [0,T_0]$~a.s.\ we can bound arbitrary moments of $\mathcal{T}_\theta$. We start by an exponential moment bound. Let $c>0$ and obtain due to Jensen's inequality
\begin{align*}
    \E e^{c  \mathcal{T}_\theta(T_0)} &\leq \frac{1}{T_0} \int_0^{T_0} \E \exp (cT_0 \mathbbm{1}_{\{|b(s)|\leq \theta\} } )~\txtd s\\
    & = \frac{1}{T_0} \int_0^{T_0} \P (|b(s)|> \theta )~\txtd s +\frac{1}{T_0} e^{cT_0} \int_0^{T_0} \P(|b(s)|\leq \theta )~\txtd s\\
    &  = 1 + \frac{e^{cT_0} -1}{T_0}\int_0^{T_0} \P (|b(s)|\leq \theta )~\txtd s.
\end{align*}
   Now since $\P( |b(s)|=0 )=0$ (since the density of the amplitude equation is not concentrated in zero, see for e.g.~\cite[Theorem 1.2, (I)]{BNT:16}) it follows by dominated convergence that 
   \begin{equation}
   \label{e:ptheta}
          p_\theta:= \int_0^{T_0} \P( |b(s)|\leq \theta )~\txtd s\to 0 \text{ as }\theta\to 0.  
   \end{equation}

   Here we also use that the law of $b$ is independent of $\theta$. Therefore letting $0<\gamma_\theta\ll 1$ 
   and applying Markov's inequality for the function $x\to e^{cx}-1$, $x\geq0$ we get
\[
\P(\mathcal{T}_\theta(T_0)\geq \gamma_\theta) 
\leq \frac{\E e^{c  \mathcal{T}_\theta(T_0)}-1}{e^{c \gamma_\theta}-1}  \leq \frac{e^{cT_0}-1}{e^{c\gamma_\theta}-1} \cdot \frac{p_\theta}{T_0}\to\frac{p_\theta }{\gamma_\theta}\quad
\text{ for } c\to 0. 
\]
Setting $\gamma_\theta=\sqrt{p_\theta}$ we obtain that
\[ 
\P(\mathcal{T}_\theta(T_0) \geq \sqrt{p_\theta})\leq\sqrt{p_\theta} . 
\]
In conclusion using $D\cF_c(b)=-cb^2$ for $T\in[2\sqrt{p_\theta},T_0]$
\[ 
\lambda_{T}(b_0) 
=\frac{1}{T} \int_0^{T} D\cF_c(b(s,\omega ))~\txtd s
\leq - \frac{c}{T}(T-\mathcal{T}_\theta(T)) \theta^{2}
\leq - \frac{c}{2}\theta^{2}
<0  \] 
on the set of large probability 
$\Omega_\theta:=\{ \mathcal{T}_\theta(T_0)<\sqrt{p}_\theta \}$, as we have $\P(\Omega_\theta)\geq 1- \sqrt{p_\theta}$.
\qed
\end{proof}


\section{Approximation of SPDEs with fractional noise via amplitude equations}\label{sec:approx}
Here we prove an approximation result for the SPDE~\eqref{spde1} which is different from the one derived in~\cite{BlNe:22}.
We recall that $T_\varepsilon=T\varepsilon^{-2}$, consider $\nu\geq 0$ and fix $\varepsilon\in(0,\varepsilon_0]$ for some $\varepsilon_0>0$ sufficiently small. 

\begin{assumptions}
    We assume that we have the upper bounds
    $\nu=\cO(\varepsilon^2)$ and $\sigma=\cO(\varepsilon^{2H+1})$.
\end{assumptions}

\begin{ansatz} The process $b$ is an $\cN$-valued process which solves the amplitude equation
\begin{align}\label{ae:1}
    \txtd b = [\nu \varepsilon^{-2} b + \cF_c(b)]~\txtd T + \sigma \varepsilon^{-2H-1}~\txtd \beta^H(T),
\end{align}
where $\beta^H(T)=\varepsilon^{2H} P_cW^H(\varepsilon^{-2}T)$ is a rescaled fractional Bronwian motion.
\end{ansatz}

\begin{theorem}\label{thm:aprox}
Let $u$ be a solution of the SPDE~\eqref{spde1} with initial condition $u_0=\cO(\varepsilon)$ in $\cH$
such that $P_su_0=\cO(\varepsilon^{2H+1})$  in $\cH$. Further, let $b$ be a solution of\eqref{ae:1} with $b(0)- \varepsilon^{-1} u_0 = \cO(\varepsilon^{2H})$. 
Then 
\begin{align}\label{error}u-\varepsilon b(\varepsilon^2\cdot) = \cO(\varepsilon^{2H+1 -\kappa} +\varepsilon^{2-\kappa})
\quad\text{ on }[0,T_\varepsilon]\text{ in }\cH \text{ for all small } \kappa>0.
\end{align}
\end{theorem}
\begin{remark}
    We notice that for $H>1/2$ the order of the error term is $\cO(\varepsilon^{2-})$, which is given by certain nonlinear terms,  whereas for $H<1/2$ we get $\cO(\varepsilon^{2H+1-})$ which arises from bounding the stochastic convolution.  
\end{remark}

\begin{proof} 
The proof is similar to and~\cite[Theorem 4.15]{BlNe:22} and~\cite[Theorem 5.3]{BlNe:23}.
Nevertheless, we repeat it here to get the  scaling of all the error terms correctly.

First of all we show that if $u_0=\cO(\varepsilon)$ in $\cH$, 
then $u=\cO(\varepsilon)$ in $\cH$ on $[0,T_\varepsilon]$ regarding that $\sigma Z=\cO(\varepsilon)$ on $[0,T_\varepsilon]$ (since $P_s Z=\cO(T^\kappa)$ and $P_cZ=\cO(T^{H})$ on $[0,T_0]$).

Using the standard transformation
$\tilde{u}:=u - \sigma Z$ we obtain the partial differential equation with random coefficients
\[
\partial_t \tilde{u} 
= A\tilde{u} + \nu (\tilde{u}+ \sigma Z) + \cF(\tilde{u}+\sigma Z). 
\]

Due to~\eqref{e:boundF} using Young's inequality we obtain the estimate 
\begin{eqnarray*}
\frac12\partial_t \|\tilde{u}\|^2 
&\leq & \nu \langle \tilde{u}+\sigma Z , \tilde{u}\rangle + \langle \cF(\tilde{u}+\sigma Z) , \tilde{u}\rangle  \\ 
&\leq & \nu \|\tilde{u}\|^2+ \nu\sigma \langle Z,\tilde{u}\rangle + C\sigma^4\|Z\|^4_X - \delta \|\tilde{u}\|^4_X \\
&\leq& \frac12(\nu\|\tilde{u}\|^2 -\delta \|\tilde{u}\|^4_X )
+ C(\nu\sigma^2 \| Z \|^2  +\sigma^4\|Z\|^4_X) \\
&\leq&  C(\nu^2+\nu\sigma^2 \| Z \|^2  +\sigma^4\|Z\|^4_X) = \cO(\varepsilon^4).
\end{eqnarray*}
Here we used that $\sigma Z=\cO(\varepsilon)$ on $[0,T_\varepsilon]$ since $\sigma=\cO(\varepsilon^{2H+1})$, where $T_\varepsilon=\cO(\varepsilon^{-2})$ and that $\nu=\cO(\varepsilon^2)$. 
This completes the proof of the first step.
Additionally, we can also conclude that 
\[
\frac12\partial_t \|\tilde{u}\|^2  = - \frac{\delta}4 \|\tilde{u}\|^4_X +  \cO(\varepsilon^4),
\]
which gives the $L^4(0,T_\varepsilon,X)$ bound on $\tilde{u}$
\begin{equation}
\label{e:bound-e-int}
 \int_0^{T_\varepsilon} \|\tilde{u}(t)\|^4_X~\txtd t 
 = \frac2{\delta}\|\tilde{u}(0)\|^2 +\cO(\varepsilon^4) 
 =\cO(\varepsilon^2). 
\end{equation}
In particular we notice that since $\sigma Z=\cO(\varepsilon)$ the estimates of $\tilde{u}$ do not depend on $H$. 

Furthermore, following the steps of the proof of~\cite[Theorem 5.3]{BlNe:23} and regarding that $P_su_0=\cO(\varepsilon^{2H+1})$ and the properties of the cubic term, we can derive that $u_s:=P_s u=\cO(\varepsilon^{2H+1})$ in $\cH$ on $[0,T_\varepsilon]$ and $\int_0^{T_\varepsilon} \|\tilde{u}_s(t)\|^4_X~\txtd t =\cO(\varepsilon^{4H+2})$, consequently $\int_0^T \|\tilde{u}_s(\varepsilon^{-2}t)\|^4_X~\txtd t = \cO(\varepsilon^{4H+4})$. For the convenience of the reader we prove these statement. To this aim we first use the splitting
\[
u=  P_cu+P_s u:= u_c+ u_s,
\] 
and define $Z_s:=P_s Z$. Again we use the standard transformation $\tilde{u}=u - \sigma Z$ so that 
\[
\tilde{u}_s=u_s-\sigma Z_s = P_s \tilde{u}. 
\] 
Thus taking the stable projection $P_s$ entails
\[
\partial_t \tilde{u}_s
= A \tilde{u}_s+\nu (\tilde{u}_s+\sigma Z_s)  +P_s\cF(u_c+\sigma Z_s+\tilde{u}_s).
\]
The Assumption~\ref{a} implies that the quadratic form of $A$ on the unit sphere of $P_s\cH$ is bounded from below by a positive constant. 
Therefore we further obtain 
\[
\frac12 \partial_t \|\tilde{u}_s\|^2
\leq  -c\| \tilde{u}_s\|^2 + \nu \|\tilde{u}_s\|^2
+ \nu \langle  \tilde{u}_s,\sigma Z_s\rangle + \langle \cF(u_c+\sigma Z_s+\tilde{u}_s),\tilde{u}_s \rangle.
\]
Using~\eqref{e:boundF} together with the fact that $\varepsilon_0$ is sufficiently small and thus $\nu=\cO(\varepsilon^2)$ is small, we derive the energy estimate

\begin{equation}
\label{e:PsuinX}
    \frac12 \partial_t \|\tilde{u}_s\|^2
\leq  -\frac{c}2 \| \tilde{u}_s\|^2 + C\nu \sigma^2 \|Z_s\|^2
 + C(\|u_c\|^4_X+ \sigma^4\|Z_s\|^4_X) 
-\delta \|\tilde{u}_s\|^4_X,
\end{equation}
for two universal constants $c,C>0$. 
Hence, via a Gronwall type estimate, for all $t\leq T_\varepsilon$
\[
\|\tilde{u}_s(t)\|^2 
\leq \|\tilde{u}_s(0)\|^2 
+ C\int_0^{t} e^{-c(t-\tau)}
\Big( \nu \sigma^2 \|Z_s\|^2
 + \|u_c\|^4_X+ \sigma^4\|Z_s\|^4_X\Big) ~\txtd \tau .
\]
We use that all norms are equivalent on $\cN$ together with the bounds $\sigma =\cO(\varepsilon^{2H+1})$, $\nu=\cO(\varepsilon^2)$, $\|Z_s\|_X=\cO(T^\kappa)$ and $P_su_0=\cO(\varepsilon^{2H+1})$
to obtain 
\[
\|\tilde{u}_s\|^2 =\cO(\varepsilon^{4H+2}) \text{ on }[0,T_\varepsilon].
\]
Thus 
\[
\|{u}_s\| 
\leq \|\tilde{u}_s\| + \sigma\|Z_s\| 
=\cO(\varepsilon^{2H+1}) + \cO(\varepsilon^{2H+1-\kappa})\text{ on }[0,T_\varepsilon]
\]
which bounds the error on $P_s\cH$. 
\begin{remark}
We notice that if $H>1/2$ the order of $\|u_s\|$ is $\cO(\varepsilon)$. This follows since $4H+2>4$ and $\|u_c\|\leq \|\tilde{u}_c\|+\sigma\|Z_c\|=\cO(\varepsilon)$ is the lower order term appearing in the integral~\eqref{e:PsuinX}. We will see in Section~\ref{six} that the error term appearing in the computation of the FTLEs will always be $\cO(\varepsilon)$ if $H>1/2$.
\end{remark}

Moreover, from the previous inequality we can infer bounds on $P_s u$ in $X$. 
From \eqref{e:PsuinX} we also obtain by integration
\[
\delta \int_0^t\|\tilde{u}_s\|^4_X~\txtd t
\leq  \|\tilde{u}_s(0) \|^2 + C \int_0^t (\nu \sigma^2 \|Z_s\|^2
 + \|u_c\|^4_X+ \|Z_s\|^4_X)~\txtd  t
\]
and thus for $H<1/2$
\[
\int_0^{T_\varepsilon}\|\tilde{u}_s(t)\|^4_X ~\txtd t = \cO(\varepsilon^{4H+2})
\qquad
\text{or} 
\qquad
\int_0^{T_0}\|\tilde{u}_s(\varepsilon^{-2}t )\|^4_X ~\txtd t = \cO(\varepsilon^{4H+4}).
\]
Again if $H>1/2$ the lowest order term is given by $\|u_c\|^4_X$
which results in 
\[
\int_0^{T_\varepsilon}\|\tilde{u}_s(t)\|^4_X ~\txtd t = \cO(\varepsilon^{4}) \text{ and therefore } \int_0^{T_0}\|\tilde{u}_s(\varepsilon^{-2}t )\|^4_X ~\txtd t = \cO(\varepsilon^{6}).
\]

\begin{remark}
  Here we notice that for a fixed time we get a better estimate for $u_s$ which does not depend on $\kappa$, which only comes from taking the supremum. We need then $Z_s(t)=\cO(1)$ for a fixed time $t$ which can be proven as in~\cite[Appendix B]{BlNe:22}.
\end{remark}

We now sketch the proof for the bound of the error term in $\cN$. First note that on $P_s\cH$ we obtain
\[
\|{u}_s\| 
\leq \|\tilde{u}_s\| + \sigma\|Z_s\| 
=\cO(\varepsilon^{2H+1}) + \cO(\varepsilon^{2H+1-\kappa})\text{ on }[0,T_\varepsilon],
\]
since $u_s=\cO(\varepsilon^{2H+1})$.
We now
 show that $\varepsilon^{-1}P_cu(\varepsilon^{-2} \cdot)-b= \cO(\varepsilon^{2H})$  on $[0,T_0]$ for our fixed $T_0$, where $u_c:=P_cu$ satisfies the SDE
\[
\txtd u_c= (\nu u_c +\cF(u_c+u_s) )\txtd t + \sigma \txtd W^H_c. 
\]
We define the error as
\[
e:=b-\varepsilon^{-1} u_c(\varepsilon^{-2}\cdot) 
\]
and obtain regarding that $W^H(t\varepsilon^{-2})=W^{H}(t)\varepsilon^{-2H}$ in law
\[
\partial_t e = \frac{\nu}{\varepsilon^2} e + P_c\cF(b) 
- P_c\cF(\varepsilon^{-1} u(\varepsilon^{-2}\cdot) ).
\]
Taking the inner product with $e$ we further get
\begin{equation}
\label{e:est-e}
\frac12 \partial_t \|e \|^2  
= \frac12 \frac{\nu}{\varepsilon^2} \|e\|^2
+ \langle \cF_c(b)- \cF_c(\varepsilon^{-1} u(\varepsilon^{-2}\cdot)) e\rangle.
\end{equation}
We use with the short-hand notation $u^{(\varepsilon)}(\cdot):=\varepsilon^{-1} u(\varepsilon^{-2}\cdot)$ and expand the cubic to derive
\begin{eqnarray*}
\lefteqn{\langle \cF_c(b)- \cF_c(u^{(\varepsilon)}), e\rangle} \\
&\leq& \langle \cF_c(b)- \cF_c(u^{(\varepsilon)}_c)  , e\rangle
+  C \|e\|  \cdot ( \|u^{(\varepsilon)}_c\|^2 \|u^{(\varepsilon)}_s\|_X +  \|u^{(\varepsilon)}_c\| \|u^{(\varepsilon)}_s\|^2_X + \|u^{(\varepsilon)}_s\|^3_X) \\
&\leq &  -\delta \|e\|^4_X +  C \|e\|  \cdot ( \|u^{(\varepsilon)}_c\|^2 \|u^{(\varepsilon)}_s\|_X +  \|u^{(\varepsilon)}_c\| \|u^{(\varepsilon)}_s\|^2_X + \|u^{(\varepsilon)}_s\|^3_X) \\
&\leq &  -\frac12\delta \|e\|^4_X +  C \|u^{(\varepsilon)}_c\|^4   + C\|u_s^{(\varepsilon)}\|^4_X.
\end{eqnarray*}
Thus from \eqref{e:est-e} we get
\[
\frac12 \partial_T \|e \|^2  
\leq \frac12 \frac{\nu}{\varepsilon^2} \|e\|^2
-\frac12\delta \|e\|^4_X +  C \|u^{(\varepsilon)}_c\|^4   + C\|u_s^{(\varepsilon)}\|^4_X.
\]
Using a Gronwall type estimate 
we obtain for $H<1/2$ and for all $T\in[0,T_0]$ with constants depending on $T_0$ 
\[
\|e(T) \|^2  \leq C \|e(0)\|^2 + C \int_0^{T} ( \|u^{(\varepsilon)}_c\|^4   + C\|u_s^{(\varepsilon)}\|^4_X)~\txtd T =\cO(\varepsilon^{4H}),
\]
using the equivalence of the norms on $\cN$, the $L^4(0,T_0,X)$-bound for $u_s$ of order $\cO(\varepsilon^{4H+4})$ and the fact that $e(0)=\cO(\varepsilon^{2H})$. This means that $\varepsilon e(\cdot)$ is of order $\cO(\varepsilon^{2H+1})$ on $[0,T_0]$ as claimed in~\eqref{error}. The case $H>1/2$ leads to an error term of order $\cO(\varepsilon)$ using that $\int_0^{T_0} \|\tilde{u}_s(\varepsilon^{-2}t)\|^4_X~\txtd t =\cO(\varepsilon^6)$. 
\end{proof}

\begin{remark}
   One can easily show that $b=\cO(1)$ using a standard comparison argument for ODEs, see~\cite[Lemma 4.10]{BlNe:22}. 
\end{remark}

\section{Lyapunov exponents for SPDEs with fractional noise}\label{six}

\subsection{Case $\nu<0$ -- Stability} 

This is the trivial case where we always have stability meaning that the FTLEs are  all negative. 

\begin{theorem}\label{l:neg}
Let Assumptions~\ref{a},\ref{f},\ref{s} hold true and let $\nu<0$. Furthermore let $u$ be a solution of~\eqref{spde} in the sense of Theorem \ref{thm:exSPDE} with deterministic initial condition $u_0\in\cH$.
Then for all $T>0$ we have with probability one
\[
\mathbb{P}(\lambda_T(u_0)\leq \nu) =1.
\]
\end{theorem}
\begin{proof}
The proof is similar to~\cite[Proposition 3.1 a)]{BlEnNe:21}. We consider a solution $v$ of the linearized problem~\eqref{linearization} around a solution $u$ of~\eqref{spde} with deterministic $\cH$-valued initial condition $u_0$. 
Recalling that $v\in H^1(0,T,\cH^{-1/2})\cap L^2(0,T,\cH^{1/2})\cap C(0,T,\cH)$ we obtain using \eqref{e:sign} the standard energy estimate
\begin{align*}
    \frac{1}{2}\partial_t \|v\|^2& =\langle A v, v\rangle +\nu \|v\|^2 +\langle D\cF(u)v, v\rangle 
    \leq \nu \|v\|^2.
\end{align*}
This implies that $ \|v(t)\| \leq \| v(0)\|e^{t\nu}$ for all $t>0$. Due to Remark \ref{propLyap}
we have for any time $T>0$
\[
\lambda_T(u_0) =\frac{1}{T}\ln (\|U_{u_0}(T)\|_{L(\cH)} )\leq  \nu 
\]
which finishes the proof.
\qed
\end{proof}
\begin{remark}
    The statement remains valid if we consider random initial data $u_0(\omega)$, in particular the random fixed point of~\eqref{spde} whose existence was established in~\cite{MaSc:04}. 
\end{remark}

\subsection{Case $1\gg \sigma \approx \nu^{1/2+H}$ --  Instability }\label{case1}
In this setting we first recall that $\sigma/\nu^{1/2+H}$ 
and  $\nu^{1/2+H}/\sigma$ are both $\cO(1)$. 
Setting $\varepsilon^2=\nu$ we obtain the amplitude equation
    \begin{align}\label{ae:case1} \txtd b = [b+\cF_c(b)]~\txtd T +\frac{\sigma}{\nu^{\frac12+H}} \txtd \beta_{\nu^H}(T). 
    \end{align}
\begin{theorem}\label{thm:main1}
    Let $b_0$ be an initial data of~\eqref{a1} for which the corresponding solution satisfies Lemma~\ref{l1}. Furthermore let $\lambda_T$ be the finite-time Lyapunov exponent of the SPDE~\eqref{spde} with initial data $u_0=\varepsilon b_0$. 
For all terminal times $T_0>0$ and all probabilities $p\in(0,1)$ 
there is a set $\Omega_p$ with probability larger than $p$ and a constant $C_p>0$
such that for $\omega\in \tilde{\Omega}\cap\Omega_p$ we have for all $T\in[0,T_0]$ that
\begin{align} 
    \lambda_{T\nu^{-1}}(\nu^{1/2} b_0) > \begin{cases} & \frac{\nu}{4}-C_p \frac{\nu^{1+H}}{T},~ \text{ if } H<1/2 \\
& \frac{\nu}{4}-C_p \frac{\nu^{3/2}}{T}, ~~\text{ if } H>1/2.
    \end{cases}
\end{align}
\end{theorem}

 The main ideas are  the approximation of the SPDE~\eqref{spde} with the amplitude equation~\eqref{a1} for $\varepsilon^2=\nu$, $\sigma=\cO(\varepsilon^{2H+1})$, Lemma~\ref{l1} and the control of the approximation error. 
We start the SPDE in $\varepsilon b_0$ and
 have that $\varepsilon b_0=\cO(\varepsilon)$ since $b=\cO(1)$. In this situation Theorem~\ref{thm:aprox} is applicable. 

Now we control the approximation error between the linearized SPDE and the linearized ODE.  To this aim we firstly introduce the slow scaling
$T=t\varepsilon^2$ and define $U$ via 
\[
u(t)= \varepsilon U(t\varepsilon^2).
\]
Let $v$ be the solution of the linearization of the SPDE around a solution $u$
\[ \partial_t v =Av+\nu v+ D\cF(u)v.
\]
On the slow scale $v(t)=\varepsilon V(t\varepsilon^2)$ 
we have (using that $D\mathcal{F}$ is quadratic)
\[
\partial_T V = \varepsilon^{-2} AV + V + D\cF(U)V. \]

Let $\varphi$ be the solution of the linearization of the amplitude equation around a solution $b$ which satisfies the support theorem
\[
\partial_T \varphi = \varphi + D\cF_c(b) \varphi.
\]
We only consider initial conditions $V(0)=\varphi(0)\in\cN$ of order $1$ independent of $\varepsilon$.

The first crucial step is the following approximation result.
\begin{theorem}\label{error:linearization1}
Let $b_0$ be an initial condition for which the corresponding solution satisfies Lemma~\ref{l1}. For any probability $p\in(0,1)$ there is a set 
 $\Omega_p$ with probability larger than $p$ such that 
the error between the linearization of the SPDE~\eqref{spde1} with initial data $u_0=\varepsilon b_0$ and of the amplitude equation~\eqref{ae:case1} is bounded by $C[\varepsilon+\varepsilon^{2H}]$.
\end{theorem}

\begin{proof} 
We show in several steps that the following error bound holds on the set of large probability $\Omega_p$ 
\begin{align}\label{error:linearization}
\|V(T) - \varphi(T) \|_{\cH} &\leq \|P_s V(T) +P_c V(T) -\varphi(T)\|
 = \cO( \varepsilon + \varepsilon^{2H}),~~T\in[0,T_0].
\end{align}

To this aim we first prove 

\begin{equation}
\label{e:Vclaim}
  \| P_sV(T)\|_{\cH} =\cO(\varepsilon) 
\quad\text{and}\quad
\| P_sV\|_{L^2(0,T_0,\cH^{1/2})} =\cO(\varepsilon^{2}).   
\end{equation} 
We first consider $V$ and use standard energy-type estimates
to obtain 
\[
 \frac12 \partial_T\|V\|^2 =  \varepsilon^{-2} \langle AV, V \rangle  + \|V\|^2 + \langle D\cF(U)V , V \rangle \\
 \leq  \|V\|^2, 
\]
where we used the non-negativity of $A$ and \eqref{e:sign}. As  $V(0)=\cO(1)$ this yields a uniform $\cO(1)$-bound on $V$ and thus $P_cV$ in $\cH$ on $[0,T_0]$ (with  constants depending on $T_0$).

We have (using the short-hand  notation $V_s:=P_sV$ and $V_c:=P_cV$) 
\begin{align}
 \frac12 \partial_T \|V_s\|^2 = & \varepsilon^{-2} \langle AV_s, V_s \rangle  + \|V_s\|^2 + \langle P_sD\cF(U)V , V_s \rangle \nonumber\\
 \leq &  - c\varepsilon^{-2}\|V_s\|^2_{\cH^{1/2}}+\|V_s\|^2 + \langle P_sD\cF(U)V_c , V_s \rangle, \label{e:Vappr}
\end{align}
where we used the spectral properties of $A$ (Assumption~\ref{a}) and the sign condition on $D\cF$ from \eqref{e:sign}. 

We now bound the nonlinear term as follows 
\[
\langle P_sD\cF(U)V_c , V_s \rangle  \leq C \|U\|_X^2 \|V_c\|_X \|V_s\|_{\cH^{\alpha}}
\leq  C \varepsilon^2 \|U\|_X^4 \|V_c\|_X^2 +\frac12 c \varepsilon^{-2}\|V_s\|^2_{\cH^{\alpha}},
\]
where we used $\varepsilon$-Young's inequality in the last step. 
Further, as shown above $V$ is $\cO(1)$ in $\cH$. 
Therefore we obtain that  $V_c$ is bounded in $X$ since all norms are equivalent on $\cN$. 
Consequently, we only need a bound on $\int_0^T \|U(S)\|_X^4~\txtd S$, which can be derived from the first step of the approximation result, Theorem~\ref{thm:aprox}. 
Namely, using that
$$ 
\int_0^{T_\varepsilon} \|u(t)\|^4_X~\txtd t =\cO(\varepsilon^{2})
$$ 
we obtain
\begin{equation}
\label{e:intBouU}
  \int_0^{T_0} \|U(S)\|^4_X~\txtd S= 
\varepsilon^2\int_0^{T_\varepsilon} \|U(t\varepsilon^{2})\|^4_X~\txtd t 
=\varepsilon^{-2} \int_0^{T_\varepsilon} \|u(t)\|^4_X~\txtd t
=\cO(1).  
\end{equation}

Thus we can conclude from \eqref{e:Vappr} for two different universal constants $c>0$ and $C>0$ using that $\|\cdot\|_{\cH^\alpha}\leq \|\cdot\|_{\cH^{1/2}}$
\begin{align}\label{second}
    \frac{1}{2}\partial_T \|V_s\|^2 &\leq -c \varepsilon^{-2} \|V_s\|^{2}_{\cH^{1/2}} +\|V_s\|^2_{\cH^\alpha} + C\varepsilon^2\|U\|^4_X \|V_c\|^2_X + \frac12\varepsilon^{-2}c\|V_s\|^2_{\cH^\alpha}\nonumber\\
    & \leq -\frac12 c\varepsilon^{-2}\|V_s\|^2_{\cH^{1/2}} + C \varepsilon^2 \|U\|^4_X \|V_c\|^2_X. 
\end{align}
Consequently, recalling that $V_s(0)=0$ via a Gronwall type estimate we obtain for all $T\in[0,T_0]$ the inequality (with constants depending on $T_0$)
\[ 
\|V_s(T)\|^2 
\leq  C  \varepsilon^2\int_0^{T_0} \|U(S)\|^4_X~\txtd S \sup_{[0,T_0]}\|V_c\|^2_X
= \cO(\varepsilon^{2}),
\]
which means that $\|V_s\|_{\cH}=\cO(\varepsilon)$, as claimed.

For the second statement in \eqref{e:Vclaim} we get from~\eqref{second}  that 
\[ c\varepsilon^{-2} \|V_s\|^2_{\cH^{1/2}} 
\leq -\frac{1}{2}\partial_T \|V_s\|^2 + C\varepsilon^2\|U\|^4_X \|V_c\|^2_X,  
\]
therefore by integration (recall $V_s(0)=0$) we derive
\[ \int_0^{T_0}\|V_s(S)\|^2_{\cH^{1/2}}~\txtd S 
\leq  -\frac{c \varepsilon^2}{2} \|V_s(T)\|^2 + C \varepsilon^4 \int_0^{T_0} \|U(S)\|^4_X~\txtd S \sup_{[0,T_0]}\|V_c\|^2_X. 
\]
As $\|V_c\|_X=\cO(1)$, $\|V_s\|_\cH=\cO(\varepsilon)$ and $\int_0^{T_0} \|U(S)\|^4_X~\txtd S=\cO(1)$ 
we obtain 
\[
\|V_s\|_{L^2(0,T_0,\cH^{1/2})}=\cO(\varepsilon^{2}).
\]
We now focus on the bound for $\|V_c-\varphi\|$. The aim is to show that the error term is of order $\cO(\varepsilon^{2H})$. We observe that $V_c-\varphi$ satisfies the equation
\[
\partial_T (V_c -\varphi) = V_c-\varphi + (D \cF_c (U) V - D\cF_c(b)\varphi),
\]
so we have to estimate
\begin{equation}
    \label{e:Vc-phi}
\frac{1}{2}\partial_T \|V_c-\varphi\|^2 =\| V_c-\varphi\|^2 + \langle D \cF_c (U) V - D\cF_c(b)\varphi, V_c-\varphi \rangle.
\end{equation}
Here the crucial term contains the nonlinearity 
\begin{eqnarray*}
  \langle D\cF_c(b)\varphi - P_c D\cF(U)V , \varphi-V_c\rangle_{\cN}
  &= &  - \langle  P_c D\cF(U)V_s ,  \varphi-P_cV\rangle_{\cN}\\
  && + \langle D\cF_c(b)\varphi - P_c D\cF(b)V_c , \varphi-P_cV\rangle_{\cN}\\
  &&+\langle  P_c [D\cF(b) - D\cF(U) ]V_c , \varphi-P_cV\rangle_{\cN},
\end{eqnarray*}
where the bound on $P_sV$ is needed in the space $X$, but the integral bounds turn out to be sufficient.
We also rely on our  $\mathcal{O}(1)$-bounds on $\varphi$ and $V_c$.

We begin with the first term above which entails
\begin{align*}
\langle  P_c D\cF(U)V_s , \varphi-V_c\rangle_{\cN} 
&\leq  C \| U\|^2_X \|V_s\|_X \|\varphi-V_c\|_{\cN}\\
 &\leq C \|U\|^2_X \|V_s\|_{\cH^{1/2}}
 \|\varphi-V_c\|_{\cN} 
\\&
\leq C  \|V_s\|^2_{\cH^{1/2}} + \|U\|^4_X
 \|\varphi-V_c\|^2_{\cN} . 
 \end{align*}
 In the last step we used again Young's inequality.

The second term gives
\[  \langle D\cF_c(b)(\varphi-V_c),\varphi-V_c \rangle_{\cN} \leq C\|b\|_{\mathcal{N}}^2 \|\varphi-V_c\|^2_{\cN}.
\]

For the last one
we use that $P_c D\cF$ and $D\cF_c$ are the same on $\cN$, which can be seen by explicitly using the properties of the cubic $\cF$.
\begin{align*}
\langle P_c [D\cF(b)-D\cF(U)] V_c,\varphi-V_c  \rangle_{\cN}
& \leq C \|b-U\|^2_X \|V_c\|_X \|\varphi-V_c\|_{\cN}
\\& \leq C \|b-U_c\|^2_{\cN} \|V_c\|_{\cN} \|\varphi-V_c\|_{\cN} 
+ C \|U_s\|^2_X \|V_c\|_{\cN} \|\varphi-V_c\|_{\cN}\\
&\leq C \|b-U_c\|^4_{\cN} \|V_c\|^2_{\cN} 
+ C \|U_s\|^4_X \|V_c\|^2_{\cN} + C\|\varphi-V_c\|^2_{\cN}.
\end{align*}
Regarding~\eqref{e:Vc-phi} and putting all the estimates together we infer that (with universal constants all denoted by $C>0$)
\begin{align*}
\frac{1}{2}\partial_T \|V_c-\varphi\|^2 &\leq \|V_c-\varphi\|^2 + C \|V_s\|^2_{\cH^{1/2}} + \|U\|^4_X \|\varphi - V_c\|^2_{\cN} +\|b\|^2_{\cN}\|\varphi-V_c\|^2_{\cN} \\
& \qquad + C \|b-U_c\|^4_{\cN} \|V_c\|^2_{\cN} + C \|U_s\|^4_{X} \|V_c\|^2_{\cN} +C \|\varphi- V_c\|^2_{\cN}\\
& \leq C \|V_c-\varphi\|^2_{\cN} ( 1+\|U\|^4_X +\|b\|^2_{\cN})
+ C \|V_s\|^2_{\cH^{1/2}} + C \|b-U_c\|^4_{\cN}\|V_c\|^2_{\cN} +C\|U_s\|^4_X \|V_c\|^2_{\cN}\\
&\leq  C \cdot  I \cdot \|V_c-\varphi\|^2_{\cN} + C \cdot J,
\end{align*}
where we set 
\[
I:= 1+\|U\|^4_X + \|b\|^2_{\cN}
\quad\text{and}\quad 
J:=\|V_s\|^2_{\cH^{1/2}} 
+  \|b-U_c\|^4_{\cN}\|V_c\|^2_{\cN}
+\|U_s\|^4_X \|V_c\|^2_{\cN}.
\]
Using Gronwall's inequality we get for $T\in[0,T_0]$
\begin{align*}
\|V_c(T)-\varphi(T)\|^2 
\leq  \Big[ \|V_c(0)-\varphi(0)\|^2 
+C\int_0^T J(S)~\txtd S \Big]\exp \Big(C\int_0^T I(S)~\txtd S\Big).
\end{align*}
We now investigate the order of $J$.
First of all, since we start the SPDE in the rescaled initial condition $u_0=\varepsilon b_0$ we obtain due to Theorem~\ref{thm:aprox} for $H<1/2$
\[
\|b(T)-U_c(T)\|_{\cH} 
= \varepsilon^{-1} \|\varepsilon b- u_c(\varepsilon^{-2} \cdot )\|_{\cH} =\cO(\varepsilon^{2H})
\]
respectively for $H>1/2$
\[
\|b(T)-U_c(T)\|_{\cH} 
= \varepsilon^{-1} \|\varepsilon b- u_c(\varepsilon^{-2} \cdot )\|_{\cH} =\cO(\varepsilon).
\]

Again we use the fact that all norms are equivalent on $\cN$.
Further, using Theorem~\ref{thm:aprox} we know that with $u_s(t)=\varepsilon U_s(t\varepsilon^2)$ 
\[ 
\int_0^{T_{0}} \|U_s(T)\|^4_X ~\txtd T
=
 \varepsilon^{-2} \int_0^{T_\varepsilon} \|u_s(t)\|^4_X~\txtd t 
=\cO(\varepsilon^{4H}).
\]
This term will  determine the order of the error $V_c-\varphi$ since $\| b(T)-U_c(T) \|^4=\cO(\varepsilon^{8H}) $, which is small only if $H<1/4$. 

Due to the above results, we have 
 pathwise bounds for 
 $\int_0^T J(S)~\txtd S$ 
 by $C(\varepsilon^2+C\varepsilon^{4H})$ 
 on a set of probability going to $1$ for $C\to\infty$. 

 Moreover, we can enlarge this set to have 
 for all $T\in[0,T_0]$
\[ \exp\Big(C\int_0^T I(S)~\txtd S \Big)
= \exp\Big(C\int_0^T (1+\|U(S)\|^4_X + \|b(S)\|^2_{\cN}~\txtd S\Big)
\leq C .
\]
Together with the previous bound this gives another condition for the set $\Omega_p$.

In summary, this entails the following error bound on $\Omega_p$
\[
\|\varphi(T)-V_c(T)\|^2 \leq C [\varepsilon^2+ \varepsilon^{4H}] ~\text{ for } T\in[0,T_0].
\]
Putting all these deliberations together, proves the statement \eqref{error:linearization} on $\Omega_p$, i.e.\
\[
\|V(T)-\varphi(T)\|_{\cH} 
\leq \|V_s(T)\|_{\cH}+\|V_c(T)-\varphi(T)\|_{\cN} 
\leq C [\varepsilon+\varepsilon^{2H}],~~T\in[0,T_0]. 
\]
Here we have to add another condition to $\Omega_p$, 
as $\|V_s\|_{\cH}\leq C\varepsilon$ uniformly in $T$ with probability going to $1$ if $C\to\infty$.
\qed
\end{proof}
\begin{remark}
    The previous computation shows that $V_s\in L^2(0,T_0;\cH^{1/2})$ has the same order as in the case of a Brownian motion. For Brownian noise, exactly this term determined the order of $J$, since all the other terms were of higher order, see~\cite{BlNe:23}. However, here
the error between $P_c V$ and $\varphi$ is now determined by the $L^4(0,T_0;X)$ bound on $U_s$. The approximation with the AE gives a term of order $\cO(\varepsilon^{4H})$ in the estimate $\|V_c(T)-\varphi(T)\|$ which becomes small only for $H<1/4$.
\end{remark}

Using this result we can proceed with the proof of Theorem~\ref{thm:main1}.
We first recall the definition of the FTLE for a solution of the SPDE starting in $u_0=\varepsilon b_0$
\begin{eqnarray*}
\lambda_{T\nu^{-1}}(\varepsilon b_0) &=& \frac{\nu}T\ln( \sup\{\|v(T/\nu)\| \ : \ \|v(0)\|=1  \})\\
&=& \frac{\nu^{3/2}}
T\ln( \sup\{\|V(T)\| \ : \ \|V(0)\|= \varepsilon^{-1}  \})\\
&=& \frac{\nu}T\ln( \sup\{\|V(T)\| \ : \ \|V(0)\|=1  \}).
\end{eqnarray*}

Using \eqref{error:linearization} for the finite-time Lyapunov exponents of the SPDE we have on $\tilde{\Omega}\cap \Omega_p$ recalling Lemma~\ref{l1}
\begin{eqnarray*}
\|V(T)\|
&\geq  & \|\varphi(T)\| - \|V(T)-\varphi(T)\| 
\geq   \|\varphi(T)\| - C [\varepsilon+\varepsilon^{2H}]\\
&\geq  &   \exp\{(1-3\delta^2)T\}  - C [\varepsilon+\varepsilon^{2H}]>0,
\end{eqnarray*}
which is positive if $\varepsilon_0$ is sufficiently small. Here we can choose $\delta=\frac{1}{2}$ as in Lemma~\ref{l1}.

To proceed we use a simple estimate for the logarithm. It is known that there exists a positive constant $c>0$ such that $ \ln(1-x)\geq -c x $ for $0\leq x\leq\frac{1}{2}.$   Therefore as $\varepsilon+\varepsilon^{2H}\ll e^{Tc} $ we have that
\begin{align*}
\ln(e^{cT}-[\varepsilon+\varepsilon^{2H}]) &=\ln(e^{cT}(1-[\varepsilon+\varepsilon^{2H}]e^{-cT})) 
= cT + \ln(1-[\varepsilon+\varepsilon^{2H}] e^{-cT})\\& \geq cT - C [\varepsilon+\varepsilon^{2H}] e^{-cT}\geq cT - C[\varepsilon+\varepsilon^{2H}] .
\end{align*}
Thus we can conclude that on $\tilde{\Omega}\cap \Omega_p$ we can bound
\begin{eqnarray*}
\lefteqn{\lambda_{T\nu^{-1}}(\varepsilon b) = 
\frac{\nu}T\ln( \sup\{\|V(T)\| \ : \ \|V(0)\|=1  \})}\\
&\geq  &  \frac{\nu}{T}\ln( \sup\{\|\varphi(T)\| - \|V(T)-\varphi(T)\| : \ \|V(0)\|=1  \})\\
&\geq  &  \frac{\nu}T\ln( \sup\{\|\varphi(T)\| - \|V(T)-\varphi(T)\| : \ \|v(0)\|=1, V(0)=\varphi(0)\in\cN  \})\\
&\geq  &  \frac{\nu}T\ln( \sup\{\|\varphi(T)\| - [\varepsilon+\varepsilon^{2H} ]  : \ \|V(0)\|=1, V(0)=\varphi(0)\in\cN  \})\\
&\geq  &   \frac{\nu}T\ln( \exp\{(T+ \int_0^T D\cF_c(b(s,\omega))~\txtd s\}  - C [\varepsilon+\varepsilon^{2H}])\\
&\geq  &   \frac{\nu}T\ln( \exp\{(1-3\delta^2)T\}  - C [\varepsilon+\varepsilon^{2H}]\\
& \geq & \nu (1-3\delta^2)-C\frac{\nu[\varepsilon+ \varepsilon^{2H}]}{T}.
\end{eqnarray*}
Choosing for e.g. $\delta=1/2$ as in Lemma~\ref{a1} proves that for $\omega\in\tilde{\Omega}\cap \Omega_p$ and for all $T\in[0,T_0]$ we have
\begin{align*}
    \lambda_{T\nu^{-1}}(\varepsilon b_0) >\begin{cases}
       & \frac{\nu}{4} -C_p\frac{\nu \varepsilon^{2H}}{T}, ~~~H<1/2\\
        & \frac{\nu}{4} -C_p \frac{\nu \varepsilon}{T},~~~~~~ H>1/2.
    \end{cases}
\end{align*}
\qed

\begin{remark}
We notice that the error term between the two linearizations is determined by $\|V_s\|_{\cH}$ (which is the same as for the Brownian motion) whereas $V_c-\varphi$ is determined by the $L^4(0,T_0;X)$ bound on $U_s$ which is of order $\cO(\varepsilon^{2H})$. If $H<1/2$ this will become small, whereas for $H>1/2$ the contribution of $V_s$ dominates. 
\end{remark}

\subsection{Case $1\gg\nu^{H+\frac{1}{2}}\gg\sigma>0$ -- Instability}\label{case2}

In this case the amplitude equation is given by 
\begin{align}\label{ae:no:noise}
\txtd b=[b+\cF_c(b)]~\txtd T. 
\end{align}
Here we consider the solution $b=0$ of the amplitude equation and let $u$ be the solution of SPDE with $u(0)=u_0=0$.
 Here we simplify the proof by neglecting the small noise $\sigma/\nu^{H+1/2}$ in the approximation. Therefore we cannot use  Theorem \ref{thm:aprox} to approximate the SPDE with~\eqref{ae:no:noise}. However all the bounds provided for $u$ in Theorem~\ref{thm:aprox} do not depend on the amplitude equation and are enough for our aims. 

As before, let $V$ be the solution of the linearized SPDE
\[ 
\partial_T V =\varepsilon^{-2}AV+ V+ D\cF(U)V
\]
and thus
\[ \partial_T V_c = V_c+  P_cD\cF(U)V = V_c+ D\cF_c(U)(V_c+V_s).
\]
The linearization of the amplitude equation around 0 reduces to 
\[ \partial_T \varphi = \varphi + DF_c(0)\varphi ,\]
which gives 
\[ \partial_T \varphi =\varphi. \] 

The main result in this case reads as follows. Recall that $T_0$ is an arbitrary terminal time and $T\in[0,T_0]$.

\begin{theorem}\label{thm:main2}
Let $\lambda_T$ be the finite-time Lyapunov exponent of the SPDE~\eqref{spde} with initial data $u_0=0$.  
For all probabilities $p\in(0,1)$ 
there is a set $\Omega_p$ with probability larger than $p$ and a constant $C_p>0$
such that for $\omega\in \Omega_p$ 
we have that  
\[  \lambda_{T\nu^{-1}}(0) > \nu - C_p\frac{\nu^{3/2}}{T},~~\text{ for all } H\in(0,1) \text{ and } T\in[0,T_0]. \]
\end{theorem}

\begin{proof}
Recall the rescaling $v(T/\nu)=\nu^{1/2}V(T)$.
Analogously to the previous case we have on a set $\Omega_p$ that
\begin{eqnarray*}
\lefteqn{\lambda_{T\nu^{-1}}(0) = \frac{\nu}T\ln( \sup\{\|v(T/\nu)\| \ : \ \|v(0)\|=1  \})}\\
&=& \frac{\nu^{3/2}}
T\ln( \sup\{\|V(T)\| \ : \ \|V(0)\|= \nu^{-1/2}  \})\\
&=& \frac{\nu}T\ln( \sup\{\|V(T)\| \ : \ \|V(0)\|=1  \})\\
&\geq  &  \frac{\nu}{T}\ln( \sup\{\|\varphi(T)\| - \|V(T)-\varphi(T)\| : \ \|V(0)\|=1  \})\\
&\geq  &  \frac{\nu}T\ln( \sup\{\|\varphi(T)\| - \|V(T)-\varphi(T)\| : \ \|v(0)\|=1, V(0)=\varphi(0)\in\cN  \})\\
&\geq  &  \frac{\nu}T\ln( \sup\{\|\varphi(T)\| - \varepsilon : \ \|V(0)\|=1, V(0)=\varphi(0)\in\cN  \})\\
& \geq & \frac{\nu}{T} \ln (\exp T -\varepsilon)  \\
&\geq & \nu  - C \frac{\nu \varepsilon }{T}.
\end{eqnarray*}
\qed
\end{proof}

Therefore we get 
\[ \partial_T (V_c-\varphi) =V_c-\varphi +D\cF_c(U)(V_c+V_s), \]
which further leads to 
\begin{align*}
    \frac{1}{2} \partial_T \|V_c-\varphi\|^2 &=\|V_c-\varphi\|^2 + \langle D\cF_c(U)(V_c+V_s),V_c-\varphi \rangle\\
    & = \|V_c-\varphi\|^2 + \langle D\cF_c(U)V_c, V_c-\varphi  \rangle +\langle D\cF_c(U)V_s,V_c-\varphi \rangle\\
    & \leq \|V_c-\varphi\|^2 
    + c\|U\|^2_X \|V_c\|_{\cN} \|V_c-\varphi\|_{\cN} 
    +c\|U\|^2_X \|V_s\|_X \|V_c-\varphi\|_{\cN}\\
    & \leq \|V_c-\varphi\|^2 
    + c\|U\|^2_X \|V_c\|_{\cN} \|V_c-\varphi\|_{\cN} 
    +c\|U\|^2_X \|V_s\|_{\cH^{1/2}} \|V_c-\varphi\|_{\cN}\\
    & \leq c \|V_c-\varphi\|^2 
    +c\|U\|^4_X \|V_c\|^2_{\cN} + c\|V_s\|^2_{\cH^{1/2}} + c\|U\|^4_X\|V_c-\varphi\|^2_\cN\\
    & \leq c(1 +\|U\|^4_X)\|V_c-\varphi\|^2_\cN +c\|U\|^4_X\|V_c\|^2_\cN + c\|V_s\|^2_{\cH^{1/2}}\\
    & \leq c\|V_c-\varphi\|^2_\cN I +  cJ,
\end{align*}
where $I:=1 + \|U\|^4_X$ and $J:=\|U\|^4_X \|V_c\|^2_\cN+ \|V_s\|^2_{\cH^{1/2}}$ and $c$ stands for a universal constant which varies from line to line. 
Again, Gronwall's inequality on $[0,T_0]$ entails
\begin{align*}
\|V_c(T)-\varphi(T)\|^2 
\leq c \Big( \|V_c(0)-\varphi(0)\|^2 +c\int_0^T J(S)~\txtd S \Big) \cdot \exp \Big(c\int_0^T I(S)~\txtd S\Big).
\end{align*}
This gives $\|V_c(T)-\varphi(T)\| \leq C\nu^{1/2}$ on a set of probability arbitrarily close to $1$, when the constant $C$ goes to $\infty$. In this case we remark that the order of $J$ does not depend on $H$ since it is determined by $\int_0^T \|U(S)\|^4~\txtd S=\cO(1)$ and $\|V_s\|_{L^2(0,T_0;\cH^{1/2})}=\cO(\nu^{1/2})$. 

Obviously, the exponent $\int_0^T I(S)~\txtd S$ can be bounded by a constant on a set of large probability. 

In conclusion we obtain on $[0,T_0]$ and on a set of probability arbitrarily close to $1$ 
\[ 
\|V(T)-\varphi(T)\|_\cH\leq  \|V_s(T)\|_{\cH} +\|V_c(T)-\varphi(T)\|_{\cH}\leq C \nu^{1/2} .
\]

\begin{remark}\label{rem:h}
    For this argument we do not need the set $\tilde\Omega$ constructed in Lemma~\ref{l1}, since we consider the linerization of the amplitude equation around zero. The result obtained here provides a bound of the error term of order $\cO(\nu^{1/2})$ independent of the value of $H\in(0,1)$. 
\end{remark}


\subsection{Case: $\nu=0$, $1\gg\sigma>0$ -- Stability at the bifurcation point}\label{case3}


At the bifurcation point, we consider $\varepsilon=\sigma^H$. Here the amplitude equation is
\[
\txtd b=P_c\cF(b)~\txtd T+\txtd \beta_\varepsilon(T).
\]
Therefore, we get 
\begin{align}\label{phi} \partial_T \varphi = D \cF_c(b)\varphi. 
\end{align}

The linearization of the SPDE~\eqref{spde} reads now as \[ \partial_t v = A v + D\cF(u)v, \]
which means that setting $v(t)=\varepsilon V(t\varepsilon^2)$ we obtain
\begin{align}\label{v} \partial_T V =\varepsilon^{-2} A V + D\cF(U)V. \end{align}
As in the previous cases we compute the error term between the two linearizations. 
\begin{theorem} Let $b_0$ be an initial datum for which the corresponding solution satisfies Lemma~\ref{l2}. 
For all $p\in(0,1)$ there is a constant $C_p$ and a set $ \Omega_p$ with probability larger than $p$ such that the approximation order between the linearization of the SPDE~\eqref{v} and of the amplitude equation~\eqref{phi} with initial data $u_0=\varepsilon b_0$  is bounded by  $C_p\varepsilon$ if $H>1/2$ respectively $C_p \varepsilon^{2H}$ if $H<1/2$ on the set $ \Omega_p$.
\end{theorem}
\begin{proof}
Since the linear term containing $\nu$ drops out, we compute new energy estimates. 
To get a $\cO(1)$ bound on $V$ we rely on the energy-estimate
\[ 
\frac{1}{2}\partial_T \|V\|^2 
\leq \varepsilon^{-2}\langle A V, V \rangle + \langle D \cF(U)V,V\rangle, 
\]
which gives now due to~\eqref{e:boundF} 
\[ 
\frac{1}{2} \partial_T \|V\|^2 
\leq \varepsilon^{-2} \langle AV, V \rangle\leq 0,
\]
due to the non-negativity of $A$. As $V(0)=\cO(1)$ this yields a uniform $\cO(1)$ bound on $V$ in $\cH$ on $[0,T_0]$.  Due to the $\cO(1)$ bound on $V$ in $\cH$ we can also bound $V_c$ in $\cN$ in any norm.  

For $V_s$ we obtain as before that $\|V_s(T)\|_{\cH}=\cO(\varepsilon)$ and that 
$\|V_s\|_{L^2(0,T_0,\cH^{1/2})}=\cO(\varepsilon^2)$. 
This follows by the usual energy estimate regarding Assumption~\ref{a} and~\eqref{e:boundF} combined with the $\varepsilon$-Young inequality. To be more precise, the estimate is based on
\begin{align*}
    \frac{1}{2} \partial_T \|V_s\|^2& = \varepsilon^{-2}\langle A V_s, V_s\rangle + \langle P_s D\cF(U)V, V_s \rangle\\
    &\leq -C \varepsilon^{-2}\|V_s\|^2_{\cH^{1/2}}
    + C \|U\|^2_X \|V_c\|_X \|V_s\|_{\cH^\alpha}\\
    & \leq - C \varepsilon^{-2} \|V_s\|^2_{\cH^{1/2}} + C \varepsilon^2 \|U\|^4_X \|V_c\|^2_X + \frac12C \varepsilon^{-2}\|V_s\|_{\cH^\alpha}^2\\
    & \leq -\frac12 C\varepsilon^{-2}\|V_s\|_{\cH^{1/2}} + C \varepsilon^2 \|U\|^4_X \|V_c\|^2_X.
\end{align*}

For $V_c$ and $\varphi$ we have 
$$ 
\partial_T V_c = D \cF_c(U)V 
\quad \text{ and } \quad 
\partial_T \varphi=D\cF_c(b) \varphi, $$
leading to
\[
\partial_T (V_c-\varphi) 
= (D\cF_c(U)- D\cF_c(b)) V + D\cF_c(b) (V-\varphi) .
\]
For the difference, we estimate as follows. Here $c$ is a universal constant which varies from line to line. 
\begin{eqnarray*}
\lefteqn{\frac{1}{2}\partial_T \|V_c-\varphi\|^2}\\ &=&\langle (D\cF_c(U) -D\cF_c(b)) V , V_c-\varphi   \rangle + \langle D\cF_c(b) (V-\varphi), V_c-\varphi \rangle\\
& =& \langle (D\cF_c(U) -D\cF_c(b)) V , V_c-\varphi   \rangle + \langle D\cF_c(b) (V_c-\varphi), V_c-\varphi  \rangle + \langle D\cF_c(b) V_s, V_c-\varphi \rangle\\
& \leq& c\|U - b\|^2_{X} (\|V_s\|_{X} +\|V_c\|_{\cN}) \|V_c-\varphi\|  
+c\|b\|^2_{\cN} \|V_c-\varphi \|^2_{\cN} 
+ c\|b\|^2_{\cN} \|V_s\|_{\cH^\alpha} \|V_c-\varphi\|
\\
& \leq& c\|U_c-b\|^2_{\cN} (\|V_s\|_{\cH^\alpha} +\|V_c\|_{\cN}) \|V_c-\varphi\|_{\cN} 
+ c\|U_s\|^2_{X} (\|V_s\|_{\cH^\alpha}
+c\|V_c\|_{\cN} ) \|V_c-\varphi\|_{\cN}
\\ 
&& +\|b\|^2_{\cN} \|V_c-\varphi\|^2_{\cN} 
 +c \|b\|^4_{\cN} \|V_s\|^2_{\cH^\alpha} 
 + c \|V_c-\varphi\|^2_{\cN}
 \\
& \leq & c \|U_c-b\|^4_{\cN} \|V_c-\varphi\|^2_{\cN} + c \|V_s\|^2_{\cH^{1/2}} 
+ c \|U_c-b\|^4_{\cN} \|V_c\|^2_{\cN} 
+ c \|V_c-\varphi\|^2_{\cN}
\\
&& + c\|U_s\|^4_X \|V_c-\varphi\|^2_{\cN} 
+  c\|U_s\|^4_X \|V_c\|^2_{\cN} 
+ c\|b\|^2_{\cN} \|V_c-\varphi\|^2_{\cN} 
 +c\|b\|^4_{\cN} \|V_s\|^2_{\cH^{1/2}}.  
\end{eqnarray*}
Thus
\[
\partial_T \|V_c-\varphi\|^2 \leq c I \|V_c-\varphi\|^2_{\cN} + c J,
\]
where 
\[I:= 1 +\|U_c-b\|^4_{\cN} +\|U_s\|^4_X +\|b\|^2_{\cN}
\] 
and 
\[J := \|V_s\|^2_{\cH^{1/2}} + \|U_c-b\|^4_{\cN} \|V_c\|^2_{\cN} +\|U_s\|^4_{X}\|V_c\|^2_{\cN} + \|b\|^4_{\cN}\|V_s\|^2_{\cH^{1/2}}.
\]
Using
Gronwall's inequality as before we obtain
\begin{align*}
\|V_c(T)-\varphi(T)\|^2 \leq \Big( \|V_c(0)-\varphi(0)\|^2 +c\int_0^T J(S)~\txtd S \Big)\cdot \exp \Big(c\int_0^T I(S)~\txtd S\Big).
\end{align*}
Now we use again the $\cO(1)$ bounds on $V_c$ and $b$ and the $\cO(\varepsilon^{4H})$-bounds for $\|U_s\|_{L^4(0,T_0,X)}$ and $\cO(\varepsilon^2)$ for
$ \|V_s\|_{L^2(0,T_0,\cH^{1/2})}$
together with Theorem~\ref{thm:aprox} that yields
\[
\|b(T)-U_c(T)\|_{\cH} 
= \varepsilon^{-1} \|\varepsilon b- u_c(\varepsilon^{-2} \cdot )\|_{\cH} =\cO(\varepsilon^{2H}).
\]
In contrast to case \ref{case1} we only need  pathwise bounds on $J$ of order $\cO(\varepsilon^2)$ respectively $\cO(\varepsilon^{4H})$ (depending on the range of $H$)
and on $b$ of order $\cO(1)$, which hold on a set of probability arbitrarily close to $1$. There is no need for $b$ being small.     

Moreover, a bound by a constant of the exponent $\int_0^T I(S)~\txtd S$ holds as before  on some set of probability arbitrarily close to $1$. 

Thus we finally conclude that
$\|V(T)-\varphi(T)\|_\cN\leq C[ \varepsilon +\varepsilon^{2H}]$ for all $T\in[0,T_0]$ on a set of probability  arbitrarily close to $1$. Actually its probability goes to $1$ if $C\to\infty$.  
\qed
\end{proof}\\

Regarding Lemma~\ref{l2} we obtain the following bound on the FTLEs. First of all we recall that for the set $\Omega_\theta=\{ \mathcal{T}_\theta(T_0)<\sqrt{p_\theta} \}$ we showed that $\P(\Omega_\theta)\geq 1 - \sqrt{p_\theta}$ where $p_\theta\to 0$
as $\theta\to 0 $. Here $T_0$ is an arbitrary terminal time and $\mathcal{T}_\theta(T_0)$ is the set of times for which the amplitude equation is smaller than $\theta$. 
Keeping this in mind and recalling that $D\cF_c(b)=-cb^2$ we derive the following statement. 

\begin{theorem}\label{thm:main3} Let $b_0$ be an initial data of~\eqref{a2} for which the corresponding solution satisfies Lemma~\ref{l2}. Furthermore let $\lambda_{\widetilde{T}}$ be the finite-time Lyapunov exponent of the SPDE~\eqref{spde} with initial condition $u_0=\varepsilon b_0$.
For all probabilities $p\in(0,1)$ 
there exist a set $\tilde\Omega_p$ with probability larger than $p$ and constants $C_p,c_p>0$ and times $T_p$ with $T_p\to0$ as $p\to 1$, $c_p\to0$ and $C_p\to\infty$ for $t\to0$ 
such that for $\omega\in \tilde\Omega_p$  we have for  $\widetilde{T}\in [{T_p},T_0]$ that
\begin{align*}
\lambda_{\widetilde{T}\varepsilon^{-1}}(\varepsilon b_0)\leq 
-c_p\varepsilon + [\varepsilon^{2}+\varepsilon^{2H+1} ] \frac{e^{c_p{\widetilde{T}} }}{\widetilde{T}} 
\end{align*}
\end{theorem}

Note that the proof relies on Lemma \ref{l2}. Having more knowledge about the density of the solution $b$ of the amplitude we conjecture it should be possible to get an $\varepsilon$-dependent $T_p$.

\begin{proof}
Our goal is to find a  bound for
\begin{eqnarray*}
\lefteqn{\lambda_{\widetilde{T}\varepsilon^{-1}}(\varepsilon b_0) = \frac{\varepsilon}{\widetilde{T}}\ln( \sup\{\|V(\widetilde{T})\| \ : \ \|V(0)\|=1  \})}\\
& \leq & \frac{\varepsilon}{\widetilde{T}}  \ln\sup (\{\|\varphi(\widetilde{T})\| + \|V(\widetilde{T})-\varphi(\widetilde{T})\| : \|V(0)\|=1\})\\
&\leq   & \frac{\varepsilon}{\widetilde{T}}\ln( \exp\{\int_0^{\widetilde{T}} D\cF_c(b(s,\omega))~\txtd s\}  +C_p[\varepsilon +\varepsilon^{2H}]).
\end{eqnarray*}
The upper bound is clear as long as the solution of the amplitude equation~\eqref{a2} does not spend too much time in zero. To exclude this possibility, we established in Lemma~\ref{l2} a lower bound on the probability of the set $\Omega_\theta$ with $\mathbb{P}(\Omega_\theta)\to1$ if $\theta\to0$, where 
\[
 \int_0^{T} D\cF_c(b(s,\omega))~\txtd s \leq -{c}_\theta T,
\]
for a constant $c_\theta=c\theta^2 \to0$ if $\theta\to0$.

This further entails that
\[ \lambda_{\widetilde{T}}(b_0) =\frac{1}{\widetilde{T}} \ln\exp\Big( \int_0^{\widetilde{T}} D\cF_c(b(s,\omega ))~\txtd s\Big) < - c_\theta<0  \] 
on the set  $\Omega_\theta$.

 Regarding this we easily derive on the set of large probability $\Omega_p\cap \Omega_\theta$ that
\begin{align*}
\ln (e^{-c_\theta\widetilde{T}} +C_p[\varepsilon +\varepsilon^{2H} ] )
& = \ln (e^{-c_\theta\widetilde{T}}(1+C_p[\varepsilon +\varepsilon^{2H}]e^{c_\theta\widetilde{T}}) ) 
=-c_\theta\widetilde{T} +\ln (1+C_p[\varepsilon+\varepsilon^{2H}] e^{c_\theta\widetilde{T}}) \\ 
&\leq - c_\theta\widetilde{T} +C_p[ \varepsilon +\varepsilon^{2H} ] e^{c_\theta\widetilde{T}}.
\end{align*}
This further leads to 
\[ \lambda_{\widetilde{T}\varepsilon^{-1}} (\varepsilon b_0)  
\leq - c_\theta\varepsilon +C_p[ \varepsilon^{2} +\varepsilon^{2H+1}] \frac{e^{c_\theta{\widetilde{T}} }}{\widetilde{T}},\]
which proves the statement.
\qed\end{proof}

\begin{remark}
    For the Brownian motion we proved a similar assertion in~\cite{BlNe:23} for the stationary solution of the amplitude equation~\ref{a2} using Birkhoff's ergodic theorem. We improve now this result in Lemma~\ref{l2} for an arbitrary solution deriving showing that the probability that the amplitude equation stays close to zero for a lot of times is small. For a higher dimensional kernel and / or multiplicative noise, this property is expected to hold, see~\cite{Bl:07} for a similar discussion for amplitude equations with Brownian motion.
\end{remark}


\subsection{Case: $1\gg\sigma\gg  \nu^{H+1/2} >0$ -- Stability}\label{case4}

This situation can be dealt with similar to Case~\ref{case3} using the amplitude equation
\begin{align}\label{a3}
\txtd \tilde{b}=\Big[\frac{\nu^{H+1/2}}{\sigma}\tilde{b} + P_c\cF(\tilde{b})\Big]~\txtd T+\txtd \beta_\varepsilon(T),
\end{align}
and its linearization
\[ \partial_T \tilde{\varphi} =\frac{\nu^{H+1/2}}{\sigma}\tilde{\varphi} +D\cF_c(\tilde{b})\tilde{\varphi}.   \]
Since the difference between~\eqref{a3} and~\eqref{a2} is of order $\cO(\frac{\nu^{H+1/2}}{\sigma})$, the following statement can be obtained analogously to Case~\ref{case4}. However there is a major difference for the error term compared to the previous cases. More precisely, here the order of $J$ will be determined by 
\[ \|\tilde{b}(T) - U_c(T)\|_{\cN} \leq   C\frac{\nu^{H+1/2}}{\sigma} +\|b(T) -U_c(T)\|_{\cN} \leq  C\frac{\nu^{H+1/2}}{\sigma} +C\varepsilon^{2H} , \]
which is in the lowest order $\frac{\nu^{H+1/2}}{\sigma}$,
since $\sigma\gg\nu^{H+1/2}$. 
\begin{theorem}\label{thm:main4}  Let $b_0$ be an initial data of~\eqref{a2} for which the corresponding solution satisfies Lemma~\ref{l2}. Furthermore let $\lambda_{\widetilde{T}}$ be the finite-time Lyapunov exponent of the SPDE~\eqref{spde} with initial data $u_0=\varepsilon b_0$.
For all probabilities $p\in(0,1)$ 
there exists a set $\tilde\Omega_p$ with probability larger than $p$, a time $T_p>0$ with $T_p\to0$ for $p\to1$ and positive constants $c_p$ and $C_p$ such that $C_p\to\infty$ and $c_p\to0$ 
such that for $\omega\in \tilde\Omega_p$ and for all $\widetilde{T}\in[T_p,T_0]$ we have 
\[ \lambda_{\widetilde{T}\sigma^{-H}} (\sigma^H b_0)
\leq \sigma^H\Big(-c_p +\frac{\nu^{H+1/2}}{\sigma} \Big) + \sigma^H\left(\frac{\nu^{H+1/2}}{\sigma}\right)^2 
\cdot \frac{e^{c_p T-\frac{\nu^{H+1/2}}{\sigma}}}{\widetilde{T}}. \]
\end{theorem}

\begin{proof}
    We only give a sketch of the proof, since this is similar to Case~\ref{case4}. Regarding the computations in Case~\ref{case3} we infer on a set of probability almost $1$ and for $T\in[0,T_0]$ that 
    \[\|V(T)-\tilde{\varphi}(T)\| \leq C \Big(\varepsilon + \Big( \varepsilon^{2H}+\frac{\nu^{H+1/2}}{\sigma}\Big)^2\Big).
\] 
This follows 
as before using for $T\in[0,T_0]$ that 
\[ 
\|V(T)-\tilde{\varphi}(T)\|_{\cH} \leq \|V_s(T)\|_{\cH} +\| V_c(T)- \varphi(T)\|_{\cN}.
\]
To estimate the last term we need a bound on $J$. As already indicated this is now determined by $\|\tilde{b}(T) -U_c(T)\|^2= \Big(\varepsilon^{2H}+\frac{\nu^{H+1/2}}{\sigma}\Big)^2$ since this expression becomes small as $\sigma\gg \nu^{H+1/2}$.
Note that in Case~\ref{case4} the order of $J$ was in lowest order determined by $\|V_s\|^2_{L^2(0,T;\cH^{1/2})}=\cO(\varepsilon^2)$ and $U_s\in L^4(0,T_0;X)=\cO(\varepsilon^{4H})$
and 
the other terms were higher order. 
In conclusion we now get for $T\in[0,T_0]$ that
\[ \|V(T)-\tilde{\varphi}(T)\|_{\cH}\leq C\Big( \varepsilon + \Big(\varepsilon^{2H}+\frac{\nu^{H+1/2}}{\sigma}\Big)^2\Big). \]
Therefore the lower bound for the FTLEs  for $\omega\in \tilde\Omega_p$ (as in Case~\ref{case4}) results in for $\widetilde{T}\in [T_p,T_0]$ 
\begin{eqnarray*}
\lefteqn{\lambda_{\widetilde{T}\varepsilon^{-1}}(\varepsilon b_0) = \frac{\varepsilon}{\widetilde{T}}\ln( \sup\{\|V(\widetilde{T})\| \ : \ \|V(0)\|=1  \})}\\
& \leq & \frac{\varepsilon}{\widetilde{T}}  \ln\sup (\{\|\varphi(\widetilde{T})\| + \|V(\widetilde{T})-\varphi(\widetilde{T})\| : \|V(0)\|=1\})\\
&\leq   & \frac{\varepsilon}{\widetilde{T}}\ln\Big( \exp\Big\{ \frac{\nu^{H+1/2}}{\sigma} \widetilde{T} +  \int_0^{\widetilde{T}} D\cF_c(b(s,\omega))~\txtd s\Big\}  +C\Big(\frac{\nu^{H+1/2}}{\sigma}\Big)^2 \Big).
\end{eqnarray*}
Using Lemma~\ref{l2} this entails for $\omega\in\tilde\Omega_p$ that 
\[ \lambda_{\widetilde{T}\varepsilon^{-1}} (\varepsilon b_0)\leq \varepsilon \Big(-c_p+\frac{\nu^{H+1/2}}{\sigma} \Big) + \varepsilon\Big(\frac{\nu^{H+1/2}}{\sigma}\Big)^2 \frac{e^{c_p-\frac{\nu^{H+1/2}}{\sigma}}}{\widetilde{T}}<0, \]
as claimed.
    \qed
\end{proof}

\end{document}